\documentclass []{amsart}
\usepackage{amsmath,amssymb}
 \usepackage{amsthm, amsfonts}
 \usepackage{enumerate}

\newtheorem{theorem}{Theorem}[section]
\newtheorem{proposition}[theorem]{Proposition}
\newtheorem{lemma}[theorem]{Lemma}
\newtheorem{corollary}[theorem]{Corollary}
\theoremstyle{definition}
\newtheorem{definition}[theorem]{Definition}

\theoremstyle{remark}
\newtheorem{remark}[theorem]{Remark}
\numberwithin{equation}{section}

\begin{document}

\title [{{THE QUASI-ZARISKI TOPOLOGY ON THE GRADED QUASI-PRIMARY SPECTRUM}}]{THE QUASI-ZARISKI TOPOLOGY ON THE GRADED QUASI-PRIMARY SPECTRUM OF A
GRADED MODULE OVER A GRADED COMMUTATIVE RING }

 \author[{{M. Jaradat and K. Al-Zoubi,  }}]{\textit{ Malik Jaradat and Khaldoun Al-Zoubi}*}

\address
{\textit{ Malik Jaradat, Department of Mathematics, The international School of Choueifat (MHS-AlDaid),
P.O.Box 66973, Alain, UAE}}
\bigskip
{\email{\textit{malikjaradat84@yahoo.com}}}

\address
{\textit{Khaldoun Al-Zoubi, Department of Mathematics and
Statistics, Jordan University of Science and Technology, P.O.Box
3030, Irbid 22110, Jordan.}}
\bigskip
{\email{\textit{kfzoubi@just.edu.jo}}}


\date{}
\begin{abstract}
 Let $G$ be a group with identity $e$. Let $R$ be a $G$-graded commutative
ring and $M$ a graded $R$-module. A proper graded submodule $Q$ of $M$ is
called a graded quasi-primary submodule if whenever $r\in h(R)$ and $m\in
h(M)$ with $rm\in Q$, then either $r\in Gr((Q:_{R}M))$ or $m\in Gr_{M}(Q)$.
The graded quasi primary spectrum $qp.Spec_{g}(M)$ is defined to be the set
of all graded quasi primary submodules of $M$. In this paper, we introduce
and study a topology on $qp.Spec_{g}(M)$, called the Quasi-Zariski Topology,
and investigate properties of this topology and some conditions under which (%
$qp.Spec_{g}(M)$, $q.\tau ^{g}$) is a Noetherian, spctral space.
\end{abstract}

\keywords{ Graded quasi-Zariski topology, graded quasi-primary spectrum, graded quasi-primary submodule. \\
$*$ Corresponding author}
 \maketitle


 \section{Introduction}

 The Zariski topology on the prime spectrum of a module over a commutative
ring have been already studied in \cite{11}. Also, the authors in \cite{20} introduced a topology on the set of all quasi-primary submodules satisfying the primeful property. The Zariski topology on the graded spectrum of graded ring in \cite{17, 18, 19} is generalized in
different ways on the graded spectrum of graded modules over graded
commutative rings as in \cite{3, 5, 9, 10, 17}. Recently, Al-Zoubi and Alkhalaf in \cite {2} introduced the concept of graded quasi-primary submodules of graded modules over graded commutative rings. Therefore these results will
be used in order to obtain the results of this paper.

Our main purpose is to study some new classes of graded modules and endow
these classes of graded submodules with the quasi-Zariski topology. In the
present work, we study the graded quasi primary submodules spectrum equipped
with the quasi-Zariski topology of a given graded $R$-module $M$, denoted by
($qp.Spec_{g}(M)$, $q.\tau ^{g}$).\newline
In section 2, we give some basic properties of graded rings and graded modules
which will be used in next sections.\newline
In section 3, for a graded $R$-module $M$, we introduce the map $\varphi
=:qp.Spec_{g}(M)\longrightarrow Spec_{g}(\overline{R})$\ where $\overline{R}%
=R/Ann(M)$, given by $\varphi (Q)=\overline{Gr((Q:_{R}M))}=\overline{%
(Gr_{M}(Q):_{R}M)}$\ for every $Q\in $\ $qp.Spec_{g}(M)$ and we investigate
some conditions under which is injective, surjective, open, and closed (see
Theorem \ref{T3.7} and Theorem \ref{T3.12}). Also, we find a base for whose
elements are quasi-compact (see Theorem \ref{T3.16}).\newline
In section 4, we introduce a quasi-Zariski topology conditions on the graded
quasi-primary submodule spectrum of a graded module such as connectedness
(see Theorem \ref{T4.1}), some equivalent conditions for ($qp.Spec_{g}(M)$, $%
q.\tau ^{g}$) to be $T_{1}$-space, irreducibility (see Theorem \ref{T4.6},
Theorem \ref{T4.7}, Theorem \ref{T4.9}, Corollary \ref{C4.10}, and Theorem %
\ref{T4.11}) and Noetheriannes (see Theorem \ref{T4.13}). Also, we study
this topology space from the point of view of spectral spaces (see Theorem %
\ref{T4.15}).
  \section{Preliminaries}

Throughout this paper all rings are commutative with identity and all
modules are unitary. First, we recall some basic properties of graded rings
which will be used in the sequel. We refer to \cite{13}{\small , }\cite{14}
and \cite{15} for these basic properties and more information on graded
rings.

 Let $G$ be a group with identity $e$. A ring $R$ is called graded, or more
precisely $G$-graded, be a if there exist a family of subgroups $\{R_{g}\}$
of $R$ such that $R=\oplus _{g\in G}R_{g}$ as abelian groups, indexed by the
elements $g\in G$, and $R_{g}R_{h}\subseteq R_{gh}$ for all $g,h\in G$. The
summands $R_{g}$ are called homogeneous components, and elements of these
summands are called homogeneous elements. If $r\in R$, then $r$ can be
written uniquely $r=\sum_{g\in G}r_{g}$, where $r_{g}$ is the component of $%
r$ in $R_{g}$. Also, $h(R)=\cup _{g\in G}R_{g}$. Let $R=\oplus _{g\in
G}R_{g} $ be a $G$-graded ring. An ideal $I$\ of $R$ is said to be a graded
ideal if $I=\oplus _{g\in G}(I\cap R_{g}):=\oplus _{g\in G}I_{g}$, where an
ideal of a graded ring need not be graded.\newline
Let $R$ be a $G$-graded ring. A proper graded ideal $I$ of $R$ is said to be
a graded prime ideal if whenever $rs\in I$, we have $r\in I$ or $s\in I$,
where $r,s\in h(R)$ (see \cite{18}). The graded radical of $I$, denoted by $%
Gr(I)$, is the set of all $r=\sum_{g\in G}r_{g}\in R$ such that for each $%
g\in G$ there exists $n_{g}\in
\mathbb{N}
$ with $r_{g}^{n_{g}}\in I$. Note that, if $r$ is a homogeneous element,
then $r\in Gr(I)$ if and only if $r^{n}\in I$ for some $n\in
\mathbb{N}
$ (see \cite{18}). It is shown in {\small \cite{18}} that $Gr(I)$ is the
intersection of all graded prime ideals of $R$ containing $I$. \newline

Let $Spec_{g}(R)$ denote the set of all graded prime ideals of $R$. For each
graded ideal $I$ of $R$, the graded variety of $I$ is the set $%
V_{R}^{g}(I)=\{P\in Spec_{g}(R)|I\subseteq P\}$. Then the set $\xi
^{g}=(R)=\{V_{R}^{g}(I)|I$ is a graded ideal of $R\}$\ satisfies the axioms
for the closed sets of a topology on $Spec_{g}(R)$, called the Zariski
topology on $Spec_{g}(R)$, denoted by ($Spec_{g}(R)$, $\tau _{R}^{g}$) (see
\cite{18, 19}).\newline
A graded ideal $I$ of $R$ is said to be a graded maximal ideal of $R$ if $%
I\neq R$ and if $J$ is a graded ideal of $R$, such that $I\subseteq
J\subseteq R$, then $I=J$ or $J=R$. Let $Max_{g}(R)$ denote the set of all
graded maximal ideals of $R$.\newline
Let $R$ be a $G$-graded ring. A proper graded ideal $I$ of $R$ is said to be
a graded quasi primary ideal if $ab\in I$ for $a$, $b$ $\in h(R)$ implies $%
a\in Gr(I)$ or $b\in Gr(I)$. Equivalently, $q$ is a graded quasi-primary
ideal of $R$ if and only if $Gr(q)$ is a graded prime ideal of $R$, (see \cite{2}). We let $%
qp.Spec_{g}(R)$ denote the set of all graded quasi-primary ideals of $R$. It
is clear that every graded prime ideal is a graded quasi-primary ideal,
i.e., $Spec_{g}(R)\subseteq qp.Spec_{g}(R)$. For a graded ideal $I$ of $R$,
the set of all graded quasi-primary ideals of $R$ containing $I$ is denoted
by $qp$-$V_{R}^{g}(I)$, i.e., $qp$-$V_{R}^{g}(I)=\{q\in
qp.Spec_{g}(R)|I\subseteq q\}$.\newline
In \cite{5, 17}, the authors are defined another variety defined for a
graded submodule $K$ of a graded $R$-module $M$. They define the variety of $%
K$ to be $V_{G}(K)=\{P\in Spec_{g}(M)\mid (P:_{R}M)\supseteq (K:_{R}M)\}$.
Then the set $\eta ^{g}(M)=\{V_{G}(K)|K$ is a graded submodule of $M\}$
contains the empty set and $Spec_{g}(M)$, satisfies the axioms for the
closed sets of a topology on $Spec_{g}(M)$, called the Zariski topology on $%
Spec_{g}(M)$, denoted by ($Spec_{g}(M)$, $\tau ^{g}$), i.e., ($Spec_{g}(M)$,
$\tau ^{g}$). The map $\psi $: $Spec_{g}(M)\longrightarrow Spec_{g}(%
\overline{R})$\ where $\overline{R}=R/Ann(M)$, defined by $\psi (P)=%
\overline{(P:_{R}M)}$\ for every $P\in $\ $Spec_{g}(M)$\ will be called the
natural map of $Spec_{g}(M)$.

\bigskip

Let $K$ be a graded submodule of a graded $R$-module $M$. We say that $K$
satisfies the graded primeful property if for each graded prime ideal $p$ of $R$
with $(K:_{R}M)\subseteq p$, there exists a graded prime submodule $P$ of $M$
containing $K$ such that $(P:_{R}M)=p$. A graded $R$-module $M$ is called
graded primeful, if either $M=(0)$ or $M\neq (0)$ and the zero graded
submodule of $M$ satisfies the graded primeful property (see \cite{1}).

Let $R$ be a $G$-graded ring and $M$ a graded $R$-module. A proper graded
submodule $Q$ of $M$ is said to be a graded quasi-primary submodule if
whenever $r\in h(R)$ and $m\in h(M)$ with $rm\in Q$, then either $r\in
Gr((Q:_{R}M))$ or $m\in Gr_{M}(Q)$ (see \cite{2}). Let $%
qp.Spec_{g}(M)$ denote the set of all graded quasi-primary submodules of $M$
satisfying the primeful property. Also if $p$ is a graded prime ideal of $R$%
, we let $qp.Spec_{g}^{p}(M)=\{Q\in qp.Spec_{g}(M)\mid Gr((Q:_{R}M))=p\}$.%
\newline
\newline
In the rest of this paper, for a graded $R$-module $M$ and for a graded
ideal $I$ of $R$, $\overline{R}=R/Ann(M)$ and $\overline{I}=I/R$ will denote
and respectively.

\section{Quasi-Zariski Topology on $qp.Spec_{g}(M)$}
In this section, we introduce the quasi-Zariski topology over the
spectrum of all graded quasi-primary submodules of a graded module $%
qp.Spec_{g}(M)$ and then investigate relationships between $qp.Spec_{g}(M)$
and $Spec_{g}(\overline{R})$.\newline
For any graded submodule $K$ of a graded $R$-module $M$, the variety for $K$%
, denoted by $qp$-$V_{M}^{g}(K),$ is the set of all graded quasi-primary
submodules $Q$ of $M$ satisfying the primeful property, such that $%
Gr((Q:_{R}M))\supseteq Gr((K:_{R}M))$.

The following Theorem shows that this variety satisfies the topology axioms
for closed sets.

\begin{theorem}
\label{T3.1}Let $M$ be a graded $R$-module. Then for graded submodules $K$, $%
N$ and $\{K_{i}\mid i\in I\}$ of $M$ we have:\newline
(1) $qp$-$V_{M}^{g}(0)=qp.Spec_{g}(M)$ and $qp$-$V_{M}^{g}(M)=\phi $.\newline
(2) $\cap _{i\in I}qp$-$V_{M}^{g}(K_{i})=qp$-$V_{M}^{g}((\sum_{i\in
I}(K_{i}:_{R}M))M)$.\newline
(3) $qp$-$V_{M}^{g}(N)\cup qp$-$V_{M}^{g}(K)=qp$-$V_{M}^{g}(N\cap K)$.
\end{theorem}

\begin{proof}
(1) $qp$-$V_{M}^{g}(0)=\{Q\in qp.Spec_{g}(M)\mid Gr((Q:_{R}M))\supseteq
Gr((0:_{R}M))\}=qp.Spec_{g}(M)$.\newline
$qp$-$V_{M}^{g}(M)=\{Q\in qp.Spec_{g}(M)\mid Gr((Q:_{R}M))\supseteq
Gr((M:_{R}M))\}=\phi $.\newline
(2) Let $Q\in \cap _{i\in I}qp$-$V_{M}^{g}(K_{i})$, it is easily verified
that $((Gr_{M}(Q):_{R}M)M:_{R}M)=(Gr_{M}(Q):_{R}M)$, thus we have%
\begin{eqnarray*}
&\Rightarrow &Gr((Q:_{R}M))\supseteq Gr((K_{i}:_{R}M))\text{, for all }i\in
I. \\
&\Rightarrow &Gr((Q:_{R}M))\supseteq (K_{i}:_{R}M)\text{, for all }i\in I%
\text{.} \\
&\Rightarrow &Gr((Q:_{R}M))\supseteq \sum\nolimits_{i\in I}(K_{i}:_{R}M)%
\text{.} \\
&\Rightarrow &Gr((Q:_{R}M))M\supseteq (\sum\nolimits_{i\in I}(K_{i}:_{R}M))M%
\text{.} \\
&\Rightarrow &(Gr((Q:_{R}M))M:_{R}M)\supseteq ((\sum\nolimits_{i\in
I}(K_{i}:_{R}M))M:_{R}M)\text{.} \\
&\Rightarrow &((Gr_{M}(Q):_{R}M)M:_{R}M)\supseteq ((\sum\nolimits_{i\in
I}(K_{i}:_{R}M))M:_{R}M)\text{. (By \cite[Theorem 3.5]{1}).} \\
&\Rightarrow &(Gr_{M}(Q):_{R}M)\supseteq ((\sum\nolimits_{i\in
I}(K_{i}:_{R}M))M:_{R}M)\text{.} \\
&\Rightarrow &Gr((Q:_{R}M))\supseteq Gr(((\sum\nolimits_{i\in
I}(K_{i}:_{R}M))M:_{R}M))\text{. (By \cite[Theorem 3.5]{1} and \cite[Lemma
2.3]{2}).} \\
&\Rightarrow &Q\in qp\text{-}V_{M}^{g}((\sum\nolimits_{i\in
I}(K_{i}:_{R}M))M)\text{.}
\end{eqnarray*}%
For the reverse inclusion let $Q\in qp$-$V_{M}^{g}((\sum_{i\in
I}(K_{i}:_{R}M))M)$%
\begin{eqnarray*}
&\Rightarrow &Gr((Q:_{R}M))\supseteq Gr(((\sum\nolimits_{i\in
I}(K_{i}:_{R}M))M:_{R}M))\text{.} \\
&\Rightarrow &Gr((Q:_{R}M))\supseteq ((\sum\nolimits_{i\in
I}(K_{i}:_{R}M))M:_{R}M)\text{.} \\
&\Rightarrow &Gr((Q:_{R}M))\supseteq ((K_{i}:_{R}M))M:_{R}M))\text{, for all
}i\in I\text{.} \\
&\Rightarrow &Gr((Q:_{R}M))\supseteq (K_{i}:_{R}M)\text{, for all }i\in I%
\text{.} \\
&\Rightarrow &Gr((Q:_{R}M))\supseteq Gr((K_{i}:_{R}M))\text{, for all }i\in I%
\text{. (By \cite[Lemma 2.3]{2}).} \\
&\Rightarrow &Q\in \cap _{i\in I}qp\text{-}V_{M}^{g}(K_{i})\text{.}
\end{eqnarray*}%
(3) From the fact that if $N\subseteq K$, for a graded submodule $N$ and $K$
of $M$, then $qp$-$V_{M}^{g}(N)\supseteq qp$-$V_{M}^{g}(K)$, we have $qp$-$%
V_{M}^{g}(N)\cup qp$-$V_{M}^{g}(K)\subseteq qp$-$V_{M}^{g}(N\cap K)$. To see
the reverse inclusion, let $Q\in qp$-$V_{M}^{g}(N\cap K)$,
\begin{eqnarray*}
&\Rightarrow &Gr((Q:_{R}M))\supseteq Gr((N\cap K:_{R}M))\text{.} \\
&\Rightarrow &Gr((Q:_{R}M))\supseteq Gr((N:_{R}M)\cap (K:_{R}M))\text{.} \\
&\Rightarrow &Gr((Q:_{R}M))\supseteq Gr((N:_{R}M))\cap Gr((K:_{R}M))\text{.
(By \cite[Lemma 2.6]{3}\textbf{)}.} \\
&\Rightarrow &Gr((Q:_{R}M))\supseteq Gr((N:_{R}M))\text{ or }%
Gr((Q:_{R}M))\supseteq Gr((K:_{R}M))\text{. (By \cite[Lemma 2.3]{2}\textbf{)}%
.} \\
&\Rightarrow &Q\in qp\text{-}V_{M}^{g}(N)\cup qp\text{-}V_{M}^{g}(K)\text{.}
\end{eqnarray*}
\end{proof}

Now we state the definition of the quasi-Zariski topology over the spectrum
of all graded quasi-primary submodules of a graded $R$-module $M$.

\begin{definition}
Let $M$ be a graded $R$-module. If $qp$-$\eta (M)$ denotes the collection of
all subsets $qp$-$V_{M}^{g}(K)$ of $qp.Spec_{g}(M)$, then $qp$-$\eta (M)$
satisfies the axioms for the closed subsets of a topological space on $%
qp.Spec_{g}(M)$. So, there exists a topology on $qp.Spec_{g}(M)$ called
quasi-Zariski topology, denoted by ($qp.Spec_{g}(M)$, $q.\tau ^{g}$).
\end{definition}

Recall that a graded $R$-module $M$ over $G$-graded ring $R$ is said to be a
graded multiplication module if for every graded submodule $K$ of $M$ there
exists a graded ideal $I$ of $R$ such that $K=IM$. It is clear that $M$ is
graded multiplication $R$-module if and only if $K=(K:_{R}M)M$ for every
graded submodule $K$ of $M$ (see \cite{6}.)\newline
Let $Y$ be a subset of $qp.Spec_{g}(M)$ for a graded $R$-module $M$. We will
denote the intersection of all elements in $Y$ by $\Im (Y)$. In the
following Proposition, we gather some basic facts about the varieties.

\begin{proposition}
\label{P3.2}Let $M$ be a graded $R$-module. Let $N$, $K$ and $\{K_{i}\mid
i\in I\}$ of $M$. Then the following hold:\newline
(1) If $N\subseteq K$, then $qp$-$V_{M}^{g}(K)\subseteq qp$-$V_{M}^{g}(N)$.%
\newline
(2) $qp$-$V_{M}^{g}(Gr_{M}(K))\subseteq qp$-$V_{M}^{g}(K)$ and equality
holds if $M$ is graded multiplication.\newline
(3) $qp$-$V_{M}^{g}(K)=qp$-$V_{M}^{g}(Gr((K:_{R}M))M)$.\newline
(4) If $Gr(N:_{R}M)=Gr((K:_{R}M))$, then $qp$-$V_{M}^{g}(N)=qp$-$%
V_{M}^{g}(K) $. The converse is also true if both $N$, $K\in qp.Spec_{g}(M)$.%
\newline
(5) $qp$-$V_{M}^{g}(K)=\bigcup\limits_{(K:_{R}M)\subseteq p\in
Spec_{g}(R)}qp.Spec_{g}^{p}(M)$.\newline
(6) Let $Y$ be a subset of $qp.Spec_{g}(M)$. Then $Y\subseteq qp$-$%
V_{M}^{g}(K)$ if and only if $Gr(K:_{R}M)\subseteq Gr(\Im (Y):_{R}M)$.
\end{proposition}

\begin{proof}
(1) Is clear.\newline
(2) $qp$-$V_{M}^{g}(Gr_{M}(K))\subseteq qp$-$V_{M}^{g}(K)$ is clearly true
by (1). The equality can be deduced from the fact $Gr_{M}(K)=Gr((K:_{R}M))M$%
, where $K$\ is a graded submodule of a graded multiplication module\textbf{%
\ }$M$ \textbf{(}\cite[Theorem 9]{16}\textbf{).}\newline
(3) Let $K$ be a proper graded submodule of $M$. Then $Q\in qp$-$%
V_{M}^{g}(K) $,%
\begin{eqnarray*}
&\Rightarrow &Gr((Q:_{R}M))M\supseteq Gr((K:_{R}M))M\text{.} \\
&\Rightarrow &Gr_{M}(Q)\supseteq Gr((K:_{R}M))M\text{.} \\
&\Rightarrow &Gr((Q:_{R}M))\supseteq (Gr((K:_{R}M))M:_{R}M)\text{. (By \cite[%
Lemma 3.5]{1}).} \\
&\Rightarrow &Gr((Q:_{R}M))\supseteq Gr((Gr((K:_{R}M))M:_{R}M))\text{. (By
\cite[Lemma 2.3]{2}).} \\
&\Rightarrow &Q\in qp\text{-}V_{M}^{g}(Gr((K:_{R}M))M)\text{.}
\end{eqnarray*}%
Thus $qp$-$V_{M}^{g}(K)\subseteq qp$-$V_{M}^{g}(Gr((K:_{R}M))M)$. For the
reverse inclusion, we have $Q\in qp$-$V_{M}^{g}(Gr((K:_{R}M))M)$,%
\begin{eqnarray*}
&\Rightarrow &Gr((Q:_{R}M)\supseteq Gr(Gr((K:_{R}M))M:_{R}M)) \\
&\Rightarrow &Gr((Q:_{R}M))\supseteq (Gr((K:_{R}M))M:_{R}M)\text{.} \\
&\Rightarrow &Gr((Q:_{R}M))\supseteq Gr((K:_{R}M)\text{.} \\
&\Rightarrow &Q\in qp\text{-}V_{M}^{g}(K)\text{.}
\end{eqnarray*}%
Finally, (4), (5) and (6) are clearly true by definitions.
\end{proof}

Now we introduce the following maps.

\begin{remark}
\label{R3.3}Let $M$ be a graded $R$-module.\newline
(1) The map $\phi ^{R}:qp.Spec_{g}(\overline{R})\rightarrow Spec_{g}(%
\overline{R})$ defined by $\phi ^{R}(\overline{q})=\overline{Gr(q)}$, is
well-defined.\newline
2. $\psi ^{q}:qp.Spec_{g}(M)\longrightarrow qp.Spec_{g}(\overline{R})$,
defined by $\psi ^{q}(Q)=$ $\overline{(Q:_{R}M)}$, is well-defined.
\end{remark}

\begin{proposition}
\label{P3.4}Let $M$ be a graded $R$-module. then we have the following
statements:\newline
(1) $(\phi ^{R})^{-1}(V_{\overline{R}}^{g}(\overline{I}))=qp$-$V_{R}^{g}(%
\overline{I})$ for every graded ideal $I$ of $R$ containing $Ann(M)$. In
particular, $\varphi ^{-1}(V_{\overline{R}}^{g}(\overline{I}))=(\phi
^{R}\circ \psi ^{q})^{-1}(V_{\overline{R}}^{g}(\overline{I}))=(\psi
^{q})^{-1}(qp$-$V_{R}^{g}(\overline{I}))$.\newline
(2) $\phi ^{R}(qp$-$V_{R}^{g}(\overline{I}))=V_{\overline{R}}^{g}(\overline{I%
})$ and $\phi ^{R}(qp.Spec_{g}(\overline{R})-qp$-$V_{R}^{g}(I))=Spec(%
\overline{R})-V_{\overline{R}}^{g}(\overline{I})$, i.e. $\phi ^{R}$ is both
closed and open.\newline
(3) $(\phi ^{M})^{-1}(V_{G}(K))=qp$-$V_{M}^{g}(K)$, for every graded
submodule $K$ of $M$,and therefore the map $\phi ^{M}$ is continuous.
\end{proposition}

\begin{proof}
(1) Let $I$ be a graded ideal of $R$ containing $Ann(M)$. Then
\begin{eqnarray*}
\overline{q} &\in &(\phi ^{R})^{-1}(V_{R}^{g}(\overline{I})) \\
&\Leftrightarrow &\phi ^{R}(\overline{q})\in V_{R}^{g}(\overline{I}) \\
&\Leftrightarrow &\overline{Gr(q)}\supseteq \overline{I} \\
&\Leftrightarrow &Gr(q)\supseteq I \\
&\Leftrightarrow &q\in qp\text{-}V_{R}^{g}(\overline{I}).\newline
\end{eqnarray*}%
(2) As we have seen in (1), $\phi ^{R}$ is a continuous map such that $(\phi
^{R})^{-1}(V_{\overline{R}}^{g}(\overline{I}))=qp$-$V_{R}^{g}(I)$ for every
graded ideal $I$ of $R$ containing $Ann(M)$. It follows that $\phi ^{R}(qp$-$%
V_{R}^{g}(\overline{I}))=\phi ^{R}(((\phi ^{R})^{-1}(V_{\overline{R}}^{g}(%
\overline{I})))=V_{\overline{R}}^{g}(\overline{I})$ as $\phi ^{R}$ is
surjective. Similarly,%
\begin{eqnarray*}
\phi ^{R}(qp.Spec_{g}(\overline{R})-qp\text{-}V_{R}^{g}(\overline{I}))
&=&\phi ^{R}((\phi ^{R})^{-1}(Spec_{g}(\overline{R}))-(\phi ^{R})^{-1}(V_{%
\overline{R}}^{g}(\overline{I})))\text{.} \\
&=&\phi ^{R}((\phi ^{R})^{-1}(Spec_{g}(\overline{R})-V_{\overline{R}}^{g}(%
\overline{I}))\text{.} \\
&=&\phi ^{R}\circ (\phi ^{R})^{-1}(Spec_{g}(\overline{R})-V_{\overline{R}%
}^{g}(\overline{I})))\text{.} \\
&=&Spec_{g}(\overline{R})-V_{\overline{R}}^{g}(\overline{I})\text{.}
\end{eqnarray*}%
(3) Suppose $(\phi ^{M})^{-1}(V_{G}(K))$. Then $\phi ^{M}(Q)\in V_{G}(K)$
and so $p=(pM(p)):_{R}M)\supseteq (K:_{R}M)$, in which $p=Gr((Q:_{R}M))$.
Hence $Gr((Q:_{R}M))\supseteq Gr((K:_{R}M))$ and so $Q\in qp$-$V_{M}^{g}(K)$%
. The argument is reversible and so $\phi ^{M}$ is continuous.
\end{proof}

\bigskip

In the next Definition we provide the natural map of $qp.Spec_{g}(M)$.

\begin{definition}
\label{D3.5}Let $M$ be a graded $R$-module. The map $\varphi
:qp.Spec_{g}(M)\longrightarrow Spec_{g}(\overline{R})$\ where $\overline{R}%
=R/Ann(M)$, defined by $\varphi (Q)=\overline{(Gr_{M}(Q):_{R}M)}$\ for every
$Q\in $\ $qp.Spec_{g}(M)$\ will be called the natural map of $qp.Spec_{g}(M)$%
.\newline
\end{definition}
\begin{remark}
\label{R3.6}Let $M$ be a graded $R$-module.\newline
(1) For every $Q\in $\ $qp.Spec_{g}(M)$, $(Gr_{M}(Q):_{R}M)$ is a graded
prime of $R$ by \cite[Lemma 2.3]{2}, and $%
(Gr_{M}(Q):_{R}M)=Gr(Q:_{R}M)$ by \cite[Theorem 3.5]{1}. Thus $\varphi (Q)=%
\overline{(Gr_{M}(Q):_{R}M)}=\overline{Gr(Q:_{R}M)}\in \overline{R}$. \newline
 (2)It is clear to see that $\varphi =(\phi ^{R}\circ \psi
^{q}):qp.Spec_{g}(M)\longrightarrow Spec_{g}(\overline{R})$\ where $%
\overline{R}=R/Ann(M)$, defined by $\varphi (Q)=\overline{Gr((Q:_{R}M))}=%
\overline{(Gr_{M}(Q):_{R}M)}$\ for every $Q\in $\ $qp.Spec_{g}(M)$.
\end{remark}

\bigskip

The following theorem provides some important characterizations about the
quasi-Zariski topology over $qp.Spec_{g}(M)$.

\begin{theorem}
\label{T3.7}Let $M$ be a graded $R$-module and let $Q$, $P\in qp.Spec_{g}(M)$%
. Then the following statements are equivalent:\newline
(1) If $qp$-$V_{M}^{g}(Q)=qp$-$V_{M}^{g}(P)$, then $Q=P$.\newline
(2) $\mid qp.Spec_{g}^{p}(M)\mid \leq 1$ for every $p\in Spec_{g}(R)$.%
\newline
(3) $\varphi $ is injective.
\end{theorem}

\begin{proof}
(1)$\Rightarrow $(2) Suppose that $Q$, $P\in qp.Spec_{g}^{p}(M)$. Then $%
(Gr_{M}(Q):_{R}M)=(Gr_{M}(P):_{R}M)=p$, and so $qp$-$V_{M}^{g}(Q)=qp$-$%
V_{M}^{g}(P)$. Thus, by the assumption (1), $Q=P$.\newline
(2)$\Rightarrow $(3) Suppose that $Q$, $P\in qp.Spec_{g}(M)$ and $\varphi
(Q)=\varphi (P)$. Then $(Gr_{M}(Q):_{R}M)=(Gr_{M}(P):_{R}M)=p$, and so $Q$, $%
P\in qp.Spec_{g}^{p}(M)$. Thus the assumption (2) implies that $Q=P$.\newline
(3)$\Rightarrow $(1) It is clear.
\end{proof}

\bigskip

The following result is easy to verify.

\begin{corollary}
\label{C3.8}Let $M$ be a graded $R$-module. If $\mid qp.Spec_{g}^{p}(M)\mid
=1$ for every $p\in Spec_{g}(R)$, then $\varphi $ is a bijective map.
\end{corollary}

\begin{proof}
By Theorem \ref{T3.7}.
\end{proof}

Now we generalize the definition of the graded primeful module to the graded
quasi-primaryful.

\begin{definition}
\label{D3.9}A graded $R$-module $M$ is called graded quasi-primaryful if
either $M=(0)$ or $M\neq (0)$ and for every $q\in qp$-$V_{R}^{g}(Ann(M))$,
there exists $Q\in qp.Spec_{g}(M)$ such that $Gr((Q:_{R}M))=Gr(q)$.
\end{definition}

\begin{remark}
\label{R3.10} Let $M$ be a graded $R$-module. Then $M$ is a graded
quasi-primaryful if and only if $\varphi $ is surjective.
\end{remark}

\bigskip

Recall that a function $\Phi $ between two topological spaces $X$ and $Y$ is
called a closed (open) map if for any closed (open) set $V$ in $X$, the
image $\Phi (V)$ is closed (open) in $Y$, (see \cite{12}\textbf{).}

\begin{theorem}
\label{T3.11}Let $M$ be a graded $R$-module. The natural map $\varphi =\phi
^{R}\circ \psi ^{q}$ is continuous with respect to the quasi-Zariski
topology; more precisely for every graded ideal $I$ of $R$ containing $%
Ann(M) $,%
\begin{equation*}
\varphi ^{-1}(V_{\overline{R}}^{g}(\overline{I}))=(\phi ^{R}\circ \psi
^{q})^{-1}(V_{\overline{R}}^{g}(\overline{I}))=(\psi ^{q})^{-1}(qp\text{-}%
V_{R}^{g}(\overline{I}))=qp-V_{M}^{g}(IM).
\end{equation*}
\end{theorem}

\begin{proof}
Suppose that\textbf{\ }$Q\in \varphi ^{-1}(V_{\overline{R}}^{g}(\overline{I}%
))$\textbf{. }Then\textbf{\ }$\varphi (Q)\in V_{\overline{R}}^{g}(I)$\textbf{%
,} and so\textbf{\ }$(Gr_{M}(Q):_{R}M)\supseteq I$\textbf{. }It follows that%
\begin{equation*}
Gr_{M}(Q)\supseteq (Gr_{M}(Q):_{R}M)M\supseteq IM\mathbf{.}
\end{equation*}%
Hence\textbf{\ }$Gr_{M}(Q)\in qp$\textbf{-}$V_{M}^{g}(IM)$\textbf{. }%
Therefore\textbf{\ }$\varphi ^{-1}(V_{\overline{R}}^{g}(\overline{I}%
))\subseteq qp$\textbf{-}$V_{M}^{g}(IM)$\textbf{. }For the reverse
inclusion, let\textbf{\ }$Q\in qp$\textbf{-}$V_{M}^{g}(IM)$\textbf{. }Then%
\begin{equation*}
\varphi (Q)=\overline{(GrM(Q):_{R}M)}\supseteq \overline{(IM:_{R}M)}%
\supseteq \overline{I}\mathbf{.}
\end{equation*}%
Hence\textbf{\ }$Q\in \varphi ^{-1}(V_{\overline{R}}^{g}(\overline{I}))$%
\textbf{.}
\end{proof}

\bigskip

\begin{theorem}
\label{T3.12}Let $M$ be a graded $R$-module and $M$ be a graded
quasi-primaryful $R$-module. If $\varphi =\phi ^{R}\circ \psi ^{q}$, then%
\begin{equation*}
\varphi (qp\text{-}V_{M}^{g}(K))=V_{\overline{R}}^{g}(\overline{Gr((K:_{R}M))%
})
\end{equation*}%
and%
\begin{equation*}
\varphi (q.Spec_{g}(M)-qp\text{-}V_{M}^{g}(K))=Spec(\overline{R})-V_{%
\overline{R}}^{g}(\overline{Gr((K:_{R}M))}),
\end{equation*}%
i.e., $\varphi $ is both closed and open.
\end{theorem}

\begin{proof}
Since $M$ is a graded quasi-primaryful, $\varphi $ is surjective. Also by Theorem \ref{T3.11}, $%
\varphi $ is a continuous map such that $\varphi ^{-1}((V_{\overline{R}}^{g}(%
\overline{I}))=qp$-$V_{M}^{g}(IM)$ for every graded ideal $I$ of $R$
containing $Ann(M)$. Hence, by Proposition \ref{P3.2}(3), for every graded
submodule $K$ of $M$,%
\begin{equation*}
\varphi ^{-1}(V_{\overline{R}}^{g}(\overline{Gr((K:_{R}M))}))=qp\text{-}%
V_{M}^{g}(Gr((K:_{R}M))M)=qp\text{-}V_{M}^{g}(K).
\end{equation*}%
Since the map $\varphi $ is surjective, we have%
\begin{equation*}
\varphi (qp\text{-}V_{M}^{g}(K))=\varphi \circ \varphi ^{-1}(V_{\overline{R}%
}^{g}(\overline{Gr((K:_{R}M))}))=V_{\overline{R}}^{g}(\overline{Gr((K:_{R}M))%
}.
\end{equation*}%
Similarly, we conclude that%
\begin{eqnarray*}
\varphi (q.Spec_{g}(M)-qp\text{-}V_{M}^{g}(K)) &=&\varphi (\varphi
^{-1}(Spec_{g}(\overline{R}))-(\varphi )^{-1}(V_{\overline{R}}^{g}(\overline{%
Gr((K:_{R}M))})) \\
&=&\varphi (\varphi ^{-1}(Spec_{g}(R)-V_{\overline{R}}^{g}(\overline{%
Gr((K:_{R}M))})) \\
&=&\varphi \circ \varphi ^{-1}(Spec_{g}(\overline{R})-V_{\overline{R}}^{g}(%
\overline{Gr((K:_{R}M))}) \\
&=&Spec_{g}(\overline{R})-V_{\overline{R}}^{g}(\overline{Gr((K:_{R}M))}).
\end{eqnarray*}
\end{proof}

\begin{corollary}
\label{T3.13}Let $M$ be a graded $R$-module. Then $\varphi =\phi ^{R}\circ
\psi ^{q}$ is bijective if and only if $\varphi $ is a graded $R$%
-homeomorphism.
\end{corollary}

\begin{proof}
By Theorem \ref{T3.11}, $\varphi $ is continuous and by Theorem \ref{T3.12},
$\varphi $ is both closed and open. Thus $\varphi $ is a bijection if and
only if $\varphi $ is a graded $R$-homeomorphism.
\end{proof}
For any graded ideal $I$ of $R$, we define $GX^{qp.R}(I)=q.Spec_{g}(R)-qp$-$%
V_{R}^{g}(I)$ as an open set of $q.Spec_{g}(I)$. Also, $%
GX_{r}^{GX^{qp.R}}=GX^{^{qp.R}}(rM)$ for any $r\in h(R)$. Clearly, $%
GX^{^{qp.R}}(0)=\phi $ and $GX^{qp.R}(1)=q.Spec_{g}(R)$.

\begin{theorem}
\label{T3.15}Let $R$ be a graded ring and $r$, $s\in h(R)$.\newline
(1) $GX_{r}^{qp.R}=\phi $ if and only if $r$ is a graded nilpotent element
of $R$.\newline
(2) $GX_{r}^{qp.R}=q.Spec_{g}(R)$ if and only if $r$ is a graded unit
element of $R$.\newline
(3) For each pair of graded ideals $I$ and $J$ of $R$, $%
GX^{qp.R}(I)=GX^{qp.R}(J)$ if and only if $Gr(I)=Gr(J)$.\newline
(4) $GX_{rs}^{qp.R}=GX_{r}^{qp.R}\cap GX_{s}^{qp.R}$.\newline
(5) ($q.Spec_{g}(R)$, $q.\tau _{R}^{g}$) is quasi-compact.\newline
(6) ($q.Spec_{g}(R)$, $q.\tau _{R}^{g}$) is a $T_{0}$-space.
\end{theorem}

\begin{proof}
(1) Let $r\in h(R)$. Then%
\begin{eqnarray*}
\phi &=&GX_{r}^{qp.R}=q.Spec_{g}(R)-qp\text{-}V_{R}^{g}(Ra)\text{.} \\
&\Leftrightarrow &qp\text{-}V_{R}^{g}(Ra)=q.Spec_{g}(R)\text{.} \\
&\Leftrightarrow &Gr(q)\supseteq Ra\text{ for every }q\in q.Spec_{g}(R)\text{%
.} \\
&\Leftrightarrow &r\text{ is in every graded prime ideal of }R\text{.} \\
&\Leftrightarrow &r\text{ is a graded nilpotent element of }R\text{.}\newline
\end{eqnarray*}%
(2) Let $r\in h(R)$. Then%
\begin{eqnarray*}
GX_{r}^{qp.R} &=&q.Spec_{g}(R)\Leftrightarrow r\notin Gr(q)\text{ for all }%
q\in q.Spec_{g}(R)\text{.} \\
&\Rightarrow &r\notin q\text{ for all }q\in Max_{g}(R)\text{.} \\
&\Rightarrow &r\text{ is graded unit.}
\end{eqnarray*}%
Conversely, if $r$ is a graded unit, then clearly $r$ is not in any graded
quasi-primary ideal. That is, $GX^{qp.R}(I)=q.Spec(R)$.\newline
(3) Suppose that $GX^{qp.R}(I)=GX^{qp.R}(J)$. Let $p$ be a graded prime
ideal of $R$ containing $I$. Since $p$ is a graded quasi-primary ideal of $R$
and $p\supseteq Gr(I)$, we have $p\in qp$-$V_{R}^{g}(I)$. Thus, by
assumption, $p\supseteq Gr(J)\supseteq J$ and so every graded prime ideal of
$R$ containing $I$ is also a graded prime ideal of $R$ containing $J$, and
vice versa. Therefore $Gr(I)=Gr(J)$. The converse is trivially true.\newline
(4) It suffices to show that $qp$-$V_{M}^{g}(Rab)=qp$-$V_{M}^{g}(Ra)\cup qp$-%
$V_{M}^{g}(Rb)$. Let $q\in qp$-$V_{M}^{g}(Rab)$. Then%
\begin{eqnarray*}
Gr(q) &\supseteq &Gr(Rab)=Gr(Ra)\cap Gr(Rb) \\
&\Leftrightarrow &Gr(q)\supseteq Gr(Ra)\text{ or }Gr(q)\supseteq Gr(Rb)\text{%
.} \\
&\Leftrightarrow &q\in qp\text{-}V_{M}^{g}(Ra)\text{ or }q\in qp\text{-}%
V_{M}^{g}(Rb)\text{.} \\
&\Leftrightarrow &q\in qp\text{-}V_{M}^{g}(Ra)\cup qp\text{-}V_{M}^{g}(Rb)%
\text{.}
\end{eqnarray*}%
\newline
(5) Let $q.Spec_{g}(R)=\cup _{i\in I}GX^{qp.R}(J_{i})$, where $%
\{J_{i}\}_{i\in I}$ is a family of graded ideals of $R$. We clearly have $%
GX_{R}^{qp.R}=q.Spec_{g}(R)=GX^{qp.R}(\sum_{i\in I}J_{i})$. Thus, by part
(3), we have $R=Gr(\sum_{i\in I}J_{i})$ and hence, $1\in \sum_{i\in I}J_{i}$%
. So there are $i_{1}$, $i_{2}$,..., $i_{n}\in I$ such that $1\in \sum_{i\in
I}^{n}J_{i_{k}}$, that is $R=\sum_{i\in I}^{n}J_{i_{k}}$. Consequently $%
q.Spec_{g}(R)=GX_{R}^{qp.R}=GX^{qp.R}(\sum_{i\in I}^{n}J_{i_{k}})=\cup
_{k=1}^{n}GX_{J_{i_{k}}}^{qp.R}.$\newline
(6) Let $q_{1}$, $q_{2}$ be two distinct points of $q.Spec_{g}(R)$. If $%
q_{1}\nsubseteq q_{2}$, then obviously $q_{2}\in GX_{q_{1}}^{qp.R}$ and $%
q_{1}\notin GX_{q_{1}}^{qp.R}$.
\end{proof}

\begin{proposition}\label{P3.16}
Let $M$ be a graded $R$-module and $r$, $s\in h(R)$.\newline
(1) $(\psi ^{q})^{-1}(GX_{\overline{r}}^{qp.\overline{R}})=GX_{r}^{qp.M}$.%
\newline
(2) $\psi ^{q}(GX_{r}^{qp.M})\subseteq GX_{\overline{r}}^{qp.\overline{R}}$
and the equality holds if $\psi ^{q}$ is surjective.\newline
(3) $GX_{rs}^{qp.M}=GX_{r}^{qp.M}\cap GX_{s}^{qp.M}$.
\end{proposition}

\begin{proof}
(1) Since $\psi ^{q}$ is continuous, by Proposition \ref{P3.4}(3), we have%
\begin{eqnarray*}
(\psi ^{q})^{-1}(GX_{\overline{r}}^{qp.\overline{R}}) &=&(\psi
^{q})^{-1}(q.Spec_{g}(\overline{R})-qp\text{-}V_{M}^{g}(\overline{a}%
\overline{R}))\text{.} \\
&=&q.Spec_{g}(M)-(\psi ^{q})^{-1}(qp\text{-}V_{M}^{g}(\overline{a}\overline{R%
}))\text{.} \\
&=&q.Spec_{g}(M)-qp\text{-}V_{M}^{g}(aM)\text{.} \\
&=&GX_{r}^{qp.M}\text{.}
\end{eqnarray*}%
(2) follows from part (1).\newline
(3) Let $r$, $s\in h(R)$. Then%
\begin{eqnarray*}
GX_{rs}^{qp.M} &=&(\psi ^{q})^{-1}(GX_{\overline{rs}}^{qp.\overline{R}})%
\text{ by part (1)\textbf{.}} \\
&=&(\psi ^{q})^{-1}(GX_{\overline{r}}^{qp.\overline{R}}\cap GX_{\overline{s}%
}^{qp.\overline{R}})\text{ by Theorem \ref{T3.15}(4)\textbf{.}} \\
&=&(\psi ^{q})^{-1}(GX_{\overline{r}}^{qp.\overline{R}})\cap (\psi
^{q})^{-1}(GX_{\overline{s}}^{qp.\overline{R}})\text{.} \\
&=&GX_{r}^{qp.M}\cap GX_{s}^{qp.M}\text{.}\newline
\end{eqnarray*}
\end{proof}

 For any graded submodule $K$ of $M$, we define $%
GX^{qp.M}(K)=q.Spec_{g}(M)-qp $-$V_{M}^{g}(K)$ as an open set of $%
q.Spec_{g}(M)$. Also, $GX_{r}^{qp.M}=GX^{qp.M}(rM)$ for any $r\in h(R)$.
Clearly, $GX^{qp.M}(0)=\phi $ and $GX^{qp.M}(1)=q.Spec_{g}(M)$. The
following result shows that the set $\beta =\{GX_{r}^{qp.M}\mid r\in h(R)\}$
is a base for the quasi-Zariski topology on $q.Spec_{g}(M)$.

\begin{theorem}
\label{T3.14}Let $M$ be a graded $R$-module. The set $\beta
=\{GX_{r}^{qp.M}\mid r\in h(R)\}$ forms a base for the quasi-Zariski
topology on $q.Spec_{g}(M)$.
\end{theorem}

\begin{proof}
We may assume that $q.Spec_{g}(M)\neq \phi $. We will show that every open
subset of ($qp.Spec_{g}(M)$, $q.\tau ^{g}$) is a union of members of $\beta $%
. Let $U$ be an open subset in ($qp.Spec_{g}(M)$, $q.\tau ^{g}$). Thus $%
U=q.Spec_{g}(M)-qp$-$V_{M}^{g}(K)$ for some graded submodule $K$ of $M$.
Therefore%
\begin{eqnarray*}
U &=&q.Spec_{g}(M)-qp\text{-}V_{M}^{g}(K) \\
&=&q.Spec_{g}(M)-qp\text{-}V_{M}^{g}(Gr((K:_{R}M))M)\text{.} \\
&=&q.Spec_{g}(M)-qp\text{-}V_{M}^{g}(\sum\limits_{r\in Gr((K:_{R}M))}rM)%
\text{.} \\
&=&q.Spec_{g}(M)-qp\text{-}V_{M}^{g}(\sum\limits_{r\in
Gr((K:_{R}M))}(rM:_{R}M)M)\text{.} \\
&=&q.Spec_{g}(M)-\bigcap\limits_{r\in Gr((K:_{R}M))}qp\text{-}V_{M}^{g}(aM)%
\text{.} \\
&=&\bigcup\limits_{r\in Gr((K:_{R}M))}GX_{r}^{qp.M}\text{.}
\end{eqnarray*}
\end{proof}
In the next Theorem gives an important characterization of the quasi-Zariski
topology over $qp.Spec_{g}(M)$ and will be need in the last Theorem of this
paper.

\begin{theorem}
\label{T3.16}Let $M$ be a graded $R$-module. If $\psi ^{q}$ is surjective,
then the open set $GX_{r}^{qp.M}$ in ($q.Spec_{g}(M)$, $q.\tau _{M}^{g}$) is
quasi-compact. In particular, the space ($q.Spec_{g}(M)$, $q.\tau _{M}^{g}$)
is quasi-compact.
\end{theorem}

\begin{proof}
Since $B =\{GX_{r}^{qp.M}\mid r\in h(R)\}$ forms a base for the
quasi-Zariski topology on $q.Spec_{g}(M)$ by Theorem \ref{T3.14}, for any
open cover of $GX_{r}^{qp.M}$, there is a family $\{r_{i}\in h(R)\mid i\in
\Delta \}$ of elements of $R$ such that $GX_{r}^{qp.M}\subseteq \cup _{i\in
I}GX_{r_{i}}^{qp.M}$. By part (2), $GX_{\overline{r}}^{qp.\overline{R}}=\psi
^{q}(GX_{r}^{qp.M})\subseteq \cup _{i\in \Delta }\psi
^{q}(GX_{r_{i}}^{qp.M})=\cup _{i\in \Delta }GX_{\overline{r_{i}}}^{qp.%
\overline{R}}$. It follows that there exists a finite subset $\Lambda $ of $%
\Delta $ such that $GX_{\overline{r}}^{qp.\overline{R}}\subseteq \cup _{i\in
\Lambda }GX_{\overline{r_{i}}}^{qp.\overline{R}}$ as $GX_{\overline{r}}^{qp.%
\overline{R}}$ is quasi-compact, since $\phi ^{R}$ is surjective, whence $%
GX_{r}^{qp.M}=(\psi ^{q})^{-1}(GX_{\overline{r}}^{qp.\overline{R}})\subseteq
$ $\cup _{i\in \Lambda }GX_{r_{i}}^{qp.M}$ by Proposition \ref{P3.16}(1).
\end{proof}

\begin{theorem}
\label{T3.17}Let $M$ be a graded $R$-module. If the map $\psi ^{q}$ is a surjective,
then the quasi-compact open sets of ($q.Spec_{g}(M)$, $q.\tau _{M}^{g}$) are
closed under finite intersection and form an open base.
\end{theorem}

\begin{proof}
It suffices to show that the intersection $U=U_{1}\cap U_{2}$ of two
quasi-compact open sets $U_{1}$ and $U_{2}$ of ($q.Spec_{g}(M)$, $q.\tau
_{M}^{g}$) is a quasi-compact set. Each $U_{j}$, $j=1$ or $2$, is a finite
union of members of the open base $\Phi =\{GX_{r}^{qp.M}\mid r\in h(R)\}$,
hence so is $U$ due to Proposition \ref{P3.16}. Put $U=\cap
_{i=1}^{n}GX_{r_{i}}^{qp.M}$ and let $\Omega $ be any open cover of $U$.
Then $\Omega $ also covers each $GX_{r_{i}}^{qp.M}$ which is quasi-compact
by Theorem \ref{T3.16}. Hence, each $GX_{r_{i}}^{qp.M}$ has a finite
subcover of $\Omega $ and so does $U$. The other part of the theorem is
trivially true due to the existence of the open base $U$.
\end{proof}

\section{Topological properties on ($qp.Spec_{g}(M)$, $q.\protect\tau ^{g}$)}

\bigskip

Let $M$ be a graded $R$-module. In this section we investigate and study
some topological properties for ($qp.Spec_{g}(M)$, $q.\tau ^{g}$) such as
the irreducibility.\newline
Recall that a topological space $(X$, $\tau )$ is said to be a connected if
it is not the union $X=X_{0}\cup X_{1}$ of two disjoint closed non-empty
subsets $X_{0}$\ and $X_{1},$ (see \cite[Definition 2.105]{*}).

\begin{remark}
\label{R4.1}
Let $(X$, $\tau _{1})$ and $(Y$, $\tau _{2})$ be two topological spaces and $%
f$ be a continuous mapping from $(X$, $\tau _{1})$ to $(Y$, $\tau _{2})$.

\ \ \  (1) If $(X$, $\tau _{1})$ is a connected (resp. quasi compact)
topological space, then $f(X)$ is a connected (resp. quasi compact)
topological space (see \cite[Theorem 2.107 and Theorem 2.138]{*}.

\ \ \ (2) For every irreducible subset $E$ of $(X$, $\tau _{1})$, $f(E)$ is
an irreducible subset of $(Y$, $\tau _{2})$, (see \cite[Proposition 2]{**}).
\end{remark}

The following Theorem give the relation between the connectedness of ($%
qp.Spec_{g}(M)$, $q.\tau ^{g}$) with other topological spaces.

\begin{theorem}
\label{T4.1}Let $M$ be a graded quasi-primaryful $R$-module. Consider the
following statements:\newline
(1) ($Spec_{g}(\overline{R})$, $\tau _{\overline{R}}^{g}$) is a connected
space.\newline
(2) ($q.Spec_{g}(\overline{R})$, $q.\tau _{\overline{R}}^{g}$) is a
connected space.\newline
(3) ($qp.Spec_{g}(M)$, $q.\tau ^{g}$) is a connected space.\newline
(4) ($Spec_{g}(M)$, $\tau ^{g}$) is a connected space.\newline
Then (1)$\Leftrightarrow $(2)$\Leftrightarrow $(3)$\Rightarrow $(4).
Moreover, if $M$ is graded primaryful, then (4)$\Rightarrow $(1).
\end{theorem}

\begin{proof}
(1)$\Leftrightarrow $(2) Suppose that ($q.Spec_{g}(R)$, $q.\tau _{R}^{g}$)
is a connected space. By Proposition \ref{P3.4}(1), the map $\phi ^{R}$ is
surjective and continuous and so ($Spec_{g}(\overline{R})$, $\tau _{%
\overline{R}}^{g}$) is also a connected space. Conversely, suppose on the
contrary that ($q.Spec_{g}(R)$, $q.\tau _{R}^{g}$) is disconnected. Then
there exists a non-empty proper subset $W$ of $q.Spec_{g}(R)$ that is both
open and closed. By Proposition \ref{P3.4}(2), $\phi ^{R}(W)$ is a non-empty
subset of $Spec_{g}(\overline{R})$ that is both open and closed. To complete
the proof, it suffices to show that $\phi ^{R}(W)$ is a proper subset of $%
Spec_{g}(\overline{R})$ that in this case ($Spec_{g}(\overline{R})$, $\tau _{%
\overline{R}}^{g}$) is disconnected, a contradiction. Since $W$ is open, $%
W=q.Spec_{g}(\overline{R})-qp$-$V_{\overline{R}}^{g}(\overline{I})$ for some
graded ideal $I$ of $R$ containing $Ann(M)$. Thus $\phi
^{R}(W)=Spec_{g}(R)-qp$-$V_{\overline{R}}^{g}(\overline{I})$ by Proposition \ref{P3.4}(2). Therefore, if $\phi ^{R}(W)=Spec_{g}(\overline{R})$, then $qp$%
-$V_{\overline{R}}^{g}(\overline{I})=\phi $, and so $\overline{I}=\overline{R%
}$, i.e., $I=R$. It follows that $W=q.Spec_{g}(\overline{R})-qp$-$V_{%
\overline{R}}^{g}(\overline{R})=q.Spec_{g}(R)$ which is impossible. Thus $%
\phi ^{R}(W)$ is a proper subset of $q.Spec_{g}(R)$\textbf{.}\newline
(3)$\Rightarrow $(1) Follows since $\varphi =\phi ^{R}\circ \psi ^{q}$ is a
surjective and continuous map of the connected space ($qp.Spec_{g}(M)$, $%
q.\tau ^{g}$).\newline
(1)$\Rightarrow $(3) Assume that ($Spec_{g}(\overline{R})$, $\tau _{%
\overline{R}}^{g}$) is connected. If ($qp.Spec_{g}(M)$, $q.\tau ^{g}$) is
disconnected, then ($qp.Spec_{g}(M)$, $q.\tau ^{g}$) must contain a
non-empty proper subset $Y$ that is both open and closed. Accordingly, $%
\varphi (Y)$ is a non-empty subset of ($Spec_{g}(\overline{R})$, $\tau _{%
\overline{R}}^{g}$) that is both open and closed by Theorem \ref{T3.12}. To
complete the proof, it suffices to show that $\varphi (Y)$ is a proper
subset of $Spec_{g}(\overline{R})$ so that ($Spec_{g}(\overline{R})$, $\tau
_{\overline{R}}^{g}$) is disconnected, a contradiction.
Since $Y$ is open, $Y=q.Spec_{g}(M)-qp$-$V_{M}^{g}(K)$ for some graded
submodule $K$ of $M$ whence $\varphi (Y)=Spec_{g}(\overline{R})-V_{\overline{%
R}}^{g}(\overline{Gr((K:_{R}M))})$ by Theorem \ref{T3.12}. Therefore, if $%
\varphi (Y)=Spec_{g}(\overline{R})$, then $V_{\overline{R}}^{g}(\overline{%
Gr((K:_{R}M))})=\phi $, and so $\overline{Gr((K:_{R}M))}=\overline{R}$,
i.e., $K=M$. It follows that $Y=q.Spec_{g}(M)-qp$-$%
V_{M}^{g}(M)=q.Spec_{g}(M) $ which is impossible. Thus $\varphi (Y)$ is a
proper subset of $Spec_{g}(\overline{R})$.\newline
(3)$\Rightarrow $(4) Assume by way of contradiction that ($Spec_{g}(M)$, $%
\tau ^{g}$) is a disconnected space. Thus $%
Spec_{g}(M)=(Spec_{g}(M)-V_{G}(K_{1}))\cup (Spec_{g}(M)-V_{G}(K_{2}))$ with
\begin{equation*}
(Spec_{g}(M)-V_{G}(K_{1}))\cap (Spec_{g}(M)-V_{G}(K_{2}))=\phi ,
\end{equation*}%
for some non-empty closed subsets $V_{G}(K_{1})$ and $V_{G}(K_{2})$ of ($%
Spec_{g}(M)$, $\tau ^{g}$). We show that%
\begin{equation*}
qp.Spec_{g}(M)=(Spec_{g}(M)-V_{G}(K_{1}))\cup (Spec_{g}(M)-V_{G}(K_{2})).
\end{equation*}%
Suppose that $Q\in qp.Spec_{g}(M)$. Since $Gr_{M}(Q)\neq M$, there is $P\in
Spec_{g}(M)$ with $P\supseteq Q$. Now since%
\begin{equation*}
(P:_{R}M)\nsupseteq (K_{i}:_{R}M)\text{ for }i=1\text{ or }i=2,
\end{equation*}%
we have $(Gr_{M}(Q):_{R}M)\nsupseteq (K_{i}:_{R}M)$ for $i=1$ or $i=2$. Thus%
\begin{equation*}
Q\in (qp.Spec_{g}(M)-qp-V_{M}^{g}(K_{1}))\cup
(qp.Spec_{g}(M)-qp-V_{M}^{g}(K_{2})).
\end{equation*}%
Moreover,%
\begin{equation*}
(qp.Spec_{g}(M)-qp-V_{M}^{g}(K_{1}))\cap
(qp.Spec_{g}(M)-qp-V_{M}^{g}(K_{2}))=\phi ,
\end{equation*}%
because if%
\begin{equation*}
Q\in (qp.Spec_{g}(M)-qp-V_{M}^{g}(K_{1}))\cap
(qp.Spec_{g}(M)-qp-V_{M}^{g}(K_{2})),
\end{equation*}%
then%
\begin{equation*}
(Gr_{M}(Q):_{R}M)\nsupseteq (K_{i}:_{R}M)\text{ for }i=1,2.
\end{equation*}%
It follows that there is $P\in qp.Spec_{g}(M)$ such that%
\begin{equation*}
(P:_{R}M)\nsupseteq (Gr_{M}(K_{i}):_{R}M)\text{ for }i=1,2.
\end{equation*}%
This implies that%
\begin{equation*}
P\in (Spec_{g}(M)-V_{G}(K_{1}))\cap (Spec_{g}(M)-V_{G}(K_{2})),
\end{equation*}%
a contradiction. Therefore ($qp.Spec_{g}(M)$, $q.\tau ^{g}$) is not a
connected space.\newline
The last statement follows from \cite[Corollary 3.8]{11}.
\end{proof}

Let $Y$ be a subset of $qp.Spec_{g}(M)$ for a graded $R$-module $M$. We will
denote the intersection of all elements in $Y$ by $\Im (Y)$ and the closure
of $Y$ in $qp.Spec_{g}(M)$ with respect to the quasi-Zariski topology by $%
cl(Y)$.

\begin{theorem}
\label{T4.2}Let $M$ be a graded $R$-module, $Y\subseteq q.Spec_{g}(M)$ and
let $Q\in q.Spec_{g}^{p}(M)$. Then we have the following:\newline
(1) $qp$-$V_{M}^{g}(\Im (Y))=cl(Y)$. In particular, $cl(\{Q\})=qp$-$%
V_{M}^{g}(Q)$.\newline
(2) If $(0)\in Y$, then $Y$ is dense in ($qp.Spec_{g}(M)$, $q.\tau ^{g}$),
that is, $cl(Y)=qp.Spec_{g}(M)$.\newline
(3) The set $\{Q\}$ is closed in ($qp.Spec_{g}(M)$, $q.\tau ^{g}$) if and
only if $p$ is a maximal element in $\{Gr((K:_{R}M))\mid K\in
q.Spec_{g}(M)\} $, and $q.Spec_{g}^{p}(M)=\{Q\}$.\newline
(4) If $\{Q\}$ is closed in ($qp.Spec_{g}(M)$, $q.\tau ^{g}$), then $Q$ is a
maximal element of $q.Spec_{g}(M)$.
\end{theorem}

\begin{proof}
(1) Suppose $L\in Y$. Then $\Im (Y)\subseteq L$. Therefore $%
Gr((L:_{R}M))\supseteq Gr((\xi (Y):_{R}M))$. Thus $L\in qp$-$V_{M}^{g}(\Im
(Y $ $))$ and so $Y\subseteq qp$-$V_{M}^{g}(\Im (Y$ $))$. Next, let $qp$-$%
V_{M}^{g}(K)$ be any closed subset of ($qp.Spec_{g}(M)$, $q.\tau ^{g}$)
containing $Y$. Then $Gr((L:_{R}M))\supseteq Gr((K:_{R}M))$ for every $L\in
Y $ so that $Gr((\Im (Y):_{R}M))\supseteq Gr((K:_{R}M))$. Hence, for every $%
P\in qp$-$V_{M}^{g}(\Im (Y))$, $Gr((P:_{R}M))\supseteq Gr((\Im
(Y):_{R}M))\supseteq Gr((K:_{R}M))$. Then $qp$-$V_{M}^{g}(\Im (Y$ $%
))\subseteq qp$-$V_{M}^{g}(K)$. Thus $qp$-$V_{M}^{g}(\Im (Y$ $))$ is the
smallest closed subset of ($qp.Spec_{g}(M)$, $q.\tau ^{g}$) containing $Y$,
hence $qp$-$V_{M}^{g}(\Im (Y$ $))=cl(Y)$.\newline
(2) Is trivial by (1)\textbf{, }for the last statement $cl(Y)=qp$-$%
V_{M}^{g}(\Im (Y))=qp$-$V_{M}^{g}(Gr_{M}(0))=qp.Spec_{g}(M)$.\newline
(3) Suppose that $\{Q\}$ is closed. Then $\{Q\}=qp$-$V_{M}^{g}(Q)$ by (1).
Let $K\in q.Spec_{g}(M)$ such that $Gr((K:_{R}M))\supseteq p=Gr((Q:_{R}M))$.
Hence, $K\in qp$-$V_{M}^{g}(Q)=\{Q\}$, and so $q.Spec_{g}^{p}(M)=\{Q\}$.
Conversely, assume that the hypothesis are hold. Let $K\in cl(\{Q\})$. Hence
by (1), $Gr((K:_{R}M))\supseteq Gr((Q:_{R}M))$. Thus, $%
Gr((K:_{R}M))=Gr((Q:_{R}M))=p$ and therefore $Q=K$. This yields $%
cl(\{Q\})=\{Q\}$.\newline
(4) Suppose $P\in q.Spec_{g}(M)$ such that $P\supseteq Q$. Then $%
Gr((P:_{R}M))\supseteq Gr((Q:_{R}M))$. i.e., $P\in qp$-$V_{M}^{g}(Q)=cl(\{Q%
\})=\{Q\}$. Hence, $P=Q$, and so $Q$ is a maximal element of $q.Spec_{g}(M)$.%
\newline
\end{proof}

A topological space $(X$, $\tau )$ is said to be a $T_{0}$-space if for each
pair of distinct points $a$, $b$ in $X$, either there exists an open set
containing $a$ and not $b$, or there exists an open set containing $b$ and
not $a$. It has been shown that a topological space is $T_{0}$ if and only
if the closures of distinct points are distinct, (see \cite{12}).

\begin{theorem}
\label{T4.3}Let $M$ be a graded $R$-module and let $Q$, $P\in qp.Spec_{g}(M)$%
. Then the following statements are equivalent:\newline
(1) ($qp.Spec_{g}(M)$, $q.\tau ^{g}$) is a $T_{0}$-space\newline
(2) If $qp$-$V_{M}^{g}(Q)=qp$-$V_{M}^{g}(P)$, then $Q=P$.\newline
(3) $\mid qp.Spec_{g}^{p}(M)\mid \leq 1$ for every $p\in Spec_{g}(R)$.%
\newline
(4) $\varphi $ is injective.
\end{theorem}

\begin{proof}
follows from Theorem \ref{T3.7} and Theorem \ref{T4.2}(1).
\end{proof}

A topological space $(X$, $\tau )$ is called a $T_{1}$-space if every
singleton set $\{x\}$ is closed in $(X$, $\tau )$. Clearly every $T_{1}$%
-space is a $T_{0}$-space, (see \cite{12}).

\begin{theorem}
\label{C4.4}Let $M$ be a graded $R$-module. Then we have the following:%
\newline
(1) ($qp.Spec_{g}(M)$, $q.\tau ^{g}$) is a $T_{1}$-space if and only if ($%
qp.Spec_{g}(M)$, $q.\tau ^{g}$) is a $T_{0}$-space and for every element $%
Q\in q.Spec_{g}(M)$, $Gr((Q:_{R}M))$ is a maximal element in $%
\{Gr((K:_{R}M))\mid K\in q.Spec_{g}(M)\}$.\newline
(2) ($qp.Spec_{g}(M)$, $q.\tau ^{g}$) is a $T_{1}$-space if and only if ($%
qp.Spec_{g}(M)$, $q.\tau ^{g}$) is a $T_{0}$-space and every graded
quasi-primary submodule of $M$ satisfying the graded primeful property is a
maximal element of $q.Spec_{g}(M)$.\newline
(3) Let $(0)\in q.Spec_{g}(M)$. Then ($qp.Spec_{g}(M)$, $q.\tau ^{g}$) is a $%
T_{1}$-space if and only if $(0)$ is the only graded quasi-primary submodule
of $M$ satisfying the graded primeful property.\newline
(4) If ($qp.Spec_{g}(M)$, $q.\tau ^{g}$) is a $T_{0}$-space and $%
qp.Spec_{g}(M)=Max_{g}(M)$ then ($qp.Spec_{g}(M)$, $q.\tau ^{g}$) is a $%
T_{1} $-space.
\end{theorem}

\begin{proof}
(1) The result is easy to check from Theorem \ref{T4.2}(3) and Theorem \ref{T4.3}.\newline
(2) The sufficiency is trivial by Theorem \ref{T4.2}(4). Conversely, suppose
$Q$, $K\in q.Spec_{g}(M)$ such that $Q\in cl(\{K\})=qp$-$V_{M}^{g}(K)$. Thus
$Gr((Q:_{R}M))\supseteq Gr((K:_{R}M))$. Since $Q$ satisfies the graded
primeful property, $Gr((Q:_{R}M))$ is a proper graded ideal of $R$ and hence
by maximality of $K$ we have $Gr((Q:_{R}M))=Gr((K:_{R}M))$, i.e., $qp$-$%
V_{M}^{g}(Q)=qp$-$V_{M}^{g}(K)$. Now, by Theorem \ref{T3.7}, we conclude
that $Q=K$. Thus $cl(\{K\})=\{K\}$, i.e., every singleton subset of ($%
qp.Spec_{g}(M)$, $q.\tau ^{g}$) is closed. So, ($qp.Spec_{g}(M)$, $q.\tau
^{g}$) is a $T_{1}$-space.\newline
(3) By part (2).\newline
(4) By part (2).
\end{proof}

\bigskip

Let $M$ be a graded $R$-module. A graded submodule $K$ of $M$ is said to
be \textit{a graded maximal submodule} if $K\neq M$\ and there is no graded
submodule $P$ of $M$ such that $K$ $\subset P\subset M$. We will denote
the set of all graded maximal submodules of $M$ by $Max_{g}(M)$, (see \cite{15}).\newline
In the next Theorem we study the relation between the $T_{1}$-space property
and the graded maximal submodules of a graded $R$-module $M$.

\begin{theorem}
\label{T4.5}Let $M$ be a finitely generated graded $R$-module. Then ($%
qp.Spec_{g}(M)$, $q.\tau ^{g}$) is a $T_{1}$-space if and only if $%
qp.Spec_{g}(M)=Max_{g}(M)$. In this case, $%
qp.Spec_{g}(M)=Spec_{g}(M)=Max_{g}(M)$.
\end{theorem}

\begin{proof}
Suppose that ($qp.Spec_{g}(M)$, $q.\tau ^{g}$) is a $T_{1}$-space. Then
every singleton subset of ($qp.Spec_{g}(M)$, $q.\tau ^{g}$) is closed.
Assume that $Q\in qp.Spec_{g}(M)$. Hence, by Theorem \ref{T4.2}(1), $qp$-$%
V_{M}^{g}(Q)=cl(\{Q\})=\{Q\}$. Since $M$ is finitely generated, there exists
$K\in Max_{g}(M)$ such that $Q\subseteq K$. It follows that $%
(Q:_{R}M)\subseteq (K:_{R}M)$ and thus $K\in qp$-$V_{M}^{g}(Q)=\{Q\}$, since
$K$ is a graded prime submodule of $M$. Hence $K=Q$, and so $Q\in Max_{g}(M)$%
. Therefore $qp.Spec_{g}(M)\subseteq Max_{g}(M)$. The reverse inclusion is
clear. Conversely, suppose that $\{Q\}$ is a singleton subset of $%
qp.Spec_{g}(M)$. If $P\in qp$-$V_{M}^{g}(Q)$, then $Gr((P:_{R}M))\supseteq
Gr((Q:_{R}M))$. Since $(Q:_{R}M)$ and $(P:_{R}M)$ are graded maximal ideals
of $R$, $(Q:_{R}M)=(P:_{R}M)$. It follows that $Q\cap P\in qp.Spec_{g}(M)$,
and so $Q\cap P\in Max_{g}(M)$. Hence $Q=P$, and so $qp$-$V_{M}^{g}(Q)=\{Q\}$%
. Therefore ($qp.Spec_{g}(M)$, $q.\tau ^{g}$) is a $T_{1}$-space. The final
claim now follows from the fact that $Max(M)\subseteq Spec_{g}(M)\subseteq
qp.Spec_{g}(M)$.
\end{proof}

\bigskip

A topological space $(X$, $\tau )$ is called irreducible if $X\neq \phi $
and if every pair of non-empty open sets in $(X$, $\tau )$ intersect. A
subset $A$ of a topological space $(X$, $\tau )$ is irreducible if for every
pair of closed subsets $A_{i}$ $(i=1$, $2)$ of $X$ with $A\subseteq
A_{1}\cup A_{2}$, we have $A\subseteq A_{1}$ or $A\subseteq A_{2}$. An
irreducible component of a topological space $A$ is a maximal irreducible
subset of $(X$, $\tau )$. A singleton subset and its closure in ($%
qp.Spec_{g}(M)$, $q.\tau ^{g}$) are both irreducible (see \cite{12}).

\begin{theorem}
\label{T4.6}Let $M$ be a graded $R$-module. Then, $qp$-$V_{M}^{g}(Q)$ is an
irreducible closed subset of ($qp.Spec_{g}(M)$, $q.\tau ^{g}$) for every
graded quasi-primary submodule $Q$ of $M$ satisfying the graded primeful
property.
\end{theorem}

\begin{proof}
By Theorem \ref{T4.2}(1).
\end{proof}

\bigskip

The next result is a good application of the quasi-Zariski topology on
graded modules. Indeed, we show a link between the irreducible closed
subsets of quasi-Zariski topology over $qp.Spec_{g}(M)$ and the graded quasi
primary submodules of a the graded $R$-module.

\begin{theorem}
\label{T4.7}Let $M$ be a graded $R$-module and $Y\subseteq q.Spec_{g}(M)$.
If $\Im (Y)$ is a graded quasi-primary submodule of $M$, then $Y$ is an
irreducible space. The converse is true, if $M$ is a graded multiplication
module and $\Im (Y)$ satisfies the graded primeful property.
\end{theorem}

\begin{proof}
Suppose $\Im (Y)$ is a graded quasi-primary submodule of $M$. Let $%
Y\subseteq Y_{1}\cup Y_{2}$ where $Y_{1}$ and $Y_{2}$ are two closed subsets
of ($qp.Spec_{g}(M)$, $q.\tau ^{g}$). Then there exist two graded submodules
$N$ and $K$ of $M$ such that $Y_{1}=qp$-$V_{M}^{g}(N)$ and $Y_{2}=qp$-$%
V_{M}^{g}(K)$. Thus, $Y\subseteq qp$-$V_{M}^{g}(N)\cup qp$-$V_{M}^{g}(K)=qp$-%
$V_{M}^{g}(N\cap K)$ and so by Proposition \ref{P3.2}(6), $Gr(((N\cap
K):_{R}M))\subseteq Gr((\Im (Y):_{R}M))$. Since $Gr((\Im (Y):_{R}M))$ is a
graded prime ideal, either $Gr((N:_{R}M))\subseteq Gr((\Im (Y):_{R}M))$ or $%
Gr((K:_{R}M))\subseteq Gr((\Im (Y):_{R}M))$. Again by using Proposition \ref%
{P3.2}(6), either $Y\subseteq qp$-$V_{M}^{g}(N)=Y_{1}$ or $Y\subseteq qp$-$%
V_{M}^{g}(K)=Y_{2}$. Thus we conclude that $Y$ is irreducible. Conversely,
assume that $M$ is a graded multiplication module and $Y$ is an irreducible
space. By the above argument, it suffices to show that $(\Im (Y):_{R}M)$ is
a graded quasi-primary ideal of $R$. Let $ab\in (\Im (Y):_{R}M)$ for some $%
a, $ $b\in h(R)$. Suppose, on the contrary, that $Ra\nsubseteq Gr((\Im
(Y):_{R}M))$ and $Rb\nsubseteq Gr((\Im (Y):_{R}M))$. Then $%
Gr((RaM:_{R}M))\nsubseteq Gr((\Im (Y):_{R}M))$ and $Gr((RbM:_{R}M))%
\nsubseteq Gr((\Im (Y):_{R}M))$. By Proposition \ref{P3.2}(6), $Y\nsubseteq
qp$-$V_{M}^{g}(RaM)$ and $Y\nsubseteq qp$-$V_{M}^{g}(RbM)$. Let $Q\in Y$.
Then $Gr((Q:_{R}M)\supseteq Gr((\Im (Y):_{R}M))\supseteq Rab$. This means
that either $RaM\subseteq Gr((Q:_{R}M)M$ or $RbM\subseteq Gr((Q:_{R}M)M$.
So, by Proposition \ref{P3.2}(1),(3), either $qp$-$V_{M}^{g}(Q)\subseteq
(RaM)$ or $qp$-$V_{M}^{g}(Q)\subseteq (RbM)$. Therefore, $Y\subseteq qp$-$%
V_{M}^{g}(RaM)\cup qp$-$V_{M}^{g}(RbM)$ and hence $Y\subseteq qp$-$%
V_{M}^{g}(RaM)$ or $Y\subseteq qp$-$V_{M}^{g}(RbM)$ as $Y$ is irreducible.
It is a contradiction.
\end{proof}

\bigskip

The following Lemma is known, but we need it here to confirm the next theorem%
{\small .}

\begin{lemma}
\cite[Theorem 3.4]{9}\label{L4.8}. Let $R$\ be a graded $G$-ring and $Y$ be
a subset of ($Spec_{g}(R)$, $\tau _{R}^{g}$). Then, $Y$\ is irreducible
subset of ($Spec_{g}(R)$, $\tau _{R}^{g}$)\ if and only if $\Im (Y)$\ is a
graded prime ideal of $R$.
\end{lemma}

\begin{proof}
Suppose that $Y$\ is an irreducible subset of ($Spec_{g}(R)$, $\tau _{R}^{g}$%
). Let $I,J$\ be a graded ideal of $R$\ such that $I\cap J\subseteq \Im (Y)$%
\ and suppose that $I\nsubseteq \Im (Y)$\ and $J\nsubseteq \Im (Y),$\ Then $%
\Im (Y)\nsubseteq V_{R}^{g}(I)$\ and $\Im (Y)\nsubseteq V_{R}^{g}(J).$\ Let $%
p\in Y$, then $I\cap J\subseteq \Im (Y)\subseteq p$. So, $p\in
V_{R}^{g}(I\cap J)=V_{R}^{g}(I)\cup V_{R}^{g}(J)$. Therefore, $Y\subseteq
V_{R}^{g}(I)\cup V_{R}^{g}(J)$\ which a contradiction to the irreducibility
of $Y$. Therefore $I\subseteq \Im (Y)$\ or $J\subseteq \Im (Y)$\ that is $%
\Im (Y)$\ is graded prime ideal by \cite[Proposition 1.2]{18}. Conversely
suppose that $Y\subseteq Spec_{g}(R)$\ such that is $\Im (Y)$\ a graded
prime ideal of $R$. Suppose that $Y\subseteq Y_{1}\cup Y_{2}$, where $%
Y_{1},Y_{2}$\ are closed subset of $Spec_{g}(R),$\ so there exist $I,J$\ are
a graded ideals of $R$, such that $Y_{1}=V_{R}^{g}(I)$\ and $%
Y_{2}=V_{R}^{g}(J).$\ Hence $Y\subseteq V_{R}^{g}(I)\cup
V_{R}^{g}(J)=V_{R}^{g}(I\cap J)$. So, $I\cap J\subseteq p,\forall p\in Y$.
Thus $I\cap J\subseteq \Im (Y)$, but $\Im (Y)$\ is graded prime, so by \cite[%
Proposition 1.2]{18} we have $I\subseteq \Im (Y)$\ or $J\subseteq \Im (Y)$.
this means that either $\Im (Y)\in V_{R}^{g}(I)$\ or $\Im (Y)\in
V_{R}^{g}(J) $. So, $Y\subseteq V_{R}^{g}(I)=Y_{1}$\ or $Y\subseteq
V_{R}^{g}(J)=Y_{2}$. Therefore, $Y$\ is irreducible.
\end{proof}

\bigskip

Let $Y$ be a closed subset of a topological space $(X$, $\tau )$. An element
$y\in Y$ is said to be a generic point of $Y$ if $Y=cl(\{y\})$. Theorem \ref%
{T4.2}(1) follows that every element $Q$ of $qp.Spec_{g}(M)$ is a generic
point of the irreducible closed subset $qp$-$V_{M}^{g}(Q)$ of ($%
qp.Spec_{g}(M)$, $q.\tau ^{g}$). Note that a generic point of a closed
subset $Y$ of a topological space is unique if the topological space is a $%
T_{0}$-space (see \cite{8}).

\begin{theorem}
\label{T4.9}Let $M$ be a graded quasi-primaryful $R$-module and $Y$ be a
subset of $q.Spec_{g}(M)$. Then $Y$ is an irreducible closed subset of ($%
qp.Spec_{g}(M)$, $q.\tau ^{g}$) if and only if $Y=qp$-$V_{M}^{g}(Q)$ for
some $Q\in q.Spec_{g}(M)$. In particular every irreducible closed subset of (%
$qp.Spec_{g}(M)$, $q.\tau ^{g}$) has a generic point.\newline
\end{theorem}

\begin{proof}
By Theorem \ref{T4.6}, $Y=qp$-$V_{M}^{g}(Q)$ is an irreducible closed subset
of ($qp.Spec_{g}(M)$, $q.\tau ^{g}$) for some $Q\in q.Spec_{g}(M)$.
Conversely, let $Y$ be an irreducible space. Hence $\varphi (Y)=Y^{^{\prime
}}$ is an irreducible subset of ($Spec_{g}(\overline{R})$, $\tau _{\overline{%
R}}^{g}$) because $\varphi $ is continuous by Theorem \ref{T3.11}. It
follows from Lemma \ref{L4.8} that $\Im (Y^{^{\prime }})=Gr((\Im (Y):_{R}M))$
is a graded prime ideal of $\overline{R}$. Therefore $Gr((\Im (Y):_{R}M))$
is a graded prime ideal of $R$. Since the map $\varphi $ is surjective,
there exists $Q\in q.Spec_{g}(M)$ such that $Gr((Q:_{R}M))=Gr((\Im
(Y):_{R}M))$. Since $Y$ is closed, there exists a graded submodule $K$ of $M$
such that $Y=qp$-$V_{M}^{g}(K)$. It means that $Gr((\Im (qp$-$%
V_{M}^{g}(K)):_{R}M))=Gr((Q:_{R}M))$ and hence $qp$-$V_{M}^{g}(\Im (Y))=qp$-$%
V_{M}^{g}(\Im (qp$-$V_{M}^{g}(K)))=qp$-$V_{M}^{g}(Q)$ by Proposition \ref%
{P3.2}(6). Thus $Y=qp$-$V_{M}^{g}(Q)$ by Theorem \ref{T4.2}(1).\newline
\end{proof}

\bigskip

The next Corollary is an application of Theorem \ref{T4.6} and Theorem \ref%
{T4.9}.

\begin{corollary}
\label{C4.10}Let $M$ be a graded $R$-module with surjective natural map $\varphi :qp.Spec_{g}(M)\longrightarrow Spec_{g}(%
\overline{R})$\ where $\overline{R}=R/Ann(M)$, given by $\varphi (Q)=%
\overline{Gr((Q:_{R}M))}=\overline{(Gr_{M}(Q):_{R}M)}$\ for every $Q\in $\ $%
qp.Spec_{g}(M)$. Then the correspondence $qp$-$V_{M}^{g}(Q)\longrightarrow
\overline{(Gr_{M}(Q):_{R}M)}$ provides a bijection from the set of
irreducible components of ($qp.Spec_{g}(M)$, $q.\tau ^{g}$) to the set of
graded minimal prime ideals of $R$.
\end{corollary}

\begin{proof}
First we show that the given correspondence is well-defined. For this, let $%
qp$-$V_{M}^{g}(Q_{1})=qp$-$V_{M}^{g}(Q_{2})$ for some $Q_{1}$, $Q_{2}\in
qp.Spec_{g}(M)$. Then $(Gr_{M}(Q_{1}):_{R}M)\supseteq (Gr_{M}(Q_{2}):_{R}M)$
and $(Gr_{M}(Q_{2}):_{R}M)\supseteq (Gr_{M}(Q_{1}):_{R}M)$. Thus, we have $%
(Gr_{M}(Q_{1}):_{R}M)=(Gr_{M}(Q_{2}):_{R}M)$. Moreover, if $qp$-$V_{M}^{g}(Q)$
is an irreducible component, $q=(Gr_{M}(Q):_{R}M)$, and $p\subseteq q$ for
some $p\in Spec_{g}(R)$, then by the surjectivity of $\varphi $, there is $%
P\in qp.Spec_{g}(M)$ such that $(Gr_{M}(P):_{R}M)=p$. It follows that $qp$-$%
V_{M}^{g}(Q)\subseteq qp$-$V_{M}^{g}(P)$. Since, by Theorem \ref{T4.6}, $qp$-%
$V_{M}^{g}(P)$ is irreducible, we have $qp$-$V_{M}^{g}(Q)=qp$-$V_{M}^{g}(P)$%
, which implies that $q=p$. Thus $p$ is a graded minimal prime ideal of $R$.
For the surjectivity of the correspondence, assume that $p$ is a graded
minimal prime ideal of $R$. Then, since $\varphi $ is surjective, there
exists $Q\in qp.Spec_{g}(M)$ such that $\overline{(Gr_{M}(Q):_{R}M)}=p$.
Moreover, by Theorem \ref{T4.6}, $qp$-$V_{M}^{g}(Q)$ is irreducible. Now let
$qp$-$V_{M}^{g}(Q)\subseteq Y$ for some irreducible subset $Y$ of ($%
qp.Spec_{g}(M)$, $q.\tau ^{g}$). Without loss of generality, we may assume
that $Y$ is closed, since the closure of an irreducible subset of ($%
qp.Spec_{g}(M)$, $q.\tau ^{g}$) is irreducible. Thus, by Theorem \ref{T4.9},
there exists $P\in qp.Spec_{g}(M)$ such that $Y=qp$-$V_{M}^{g}(P)$. It
follows that $p=$ $\overline{(Gr_{M}(Q):_{R}M)}\supseteq $ $\overline{%
(Gr_{M}(P):_{R}M)}$, and so, by the minimality of $p$, we have $p=$ $%
\overline{(Gr_{M}(Q):_{R}M)}=$ $\overline{(Gr_{M}(P):_{R}M)}$. Hence $qp$-$%
V_{M}^{g}(Q)=qp$-$V_{M}^{g}(P)=Y$. This means that $qp$-$V_{M}^{g}(Q)$ is an
irreducible component of ($qp.Spec_{g}(M)$, $q.\tau ^{g}$).
\end{proof}

\bigskip

\begin{theorem}
\label{T4.11}Let $M$ be a graded quasi-primaryful $R$-module. The set of all
irreducible components of ($qp.Spec_{g}(M)$, $q.\tau ^{g}$) is of the form $%
\Phi =\{qp$-$V_{M}^{g}(Gr(q)M)\mid q\in qp$-$V_{R}^{g}(Ann(M))$ and $Gr(q)$
is a minimal element of $V_{R}^{g}(Ann(M))$ with respect to inclusion$\}$.
\end{theorem}

\begin{proof}
Suppose $Y$ is an irreducible component of ($qp.Spec_{g}(M)$, $q.\tau ^{g}$%
). By Theorem \ref{T4.9}, $Y=qp$-$V_{M}^{g}(Q)$ for some $Q\in q.Spec_{g}(M)$%
. Hence, $Y=qp$-$V_{M}^{g}(Q)=qp$-$V_{M}^{g}(Gr((Q:_{R}M))M)$ by Proposition %
\ref{P3.2}(3). Let $q=(Q:_{R}M)$. Now, it suffices to show that $Gr(q)$ is a
minimal element of $V_{R}^{g}(Ann(M))$ with respect to inclusion. To see
this let $q^{^{\prime }}\in V_{R}^{g}(Ann(M))$ and $q^{^{\prime }}\subseteq
Gr(q)$. Then there exists an element $Q^{^{\prime }}\in q.Spec_{g}(M)$ such
that $Gr((Q^{^{\prime }}:_{R}M))=q^{^{\prime }}$ because $M$ is graded
quasi-primaryful. So, $Y=qp$-$V_{M}^{g}(Q)\subseteq qp$-$V_{M}^{g}(Q^{^{%
\prime }})$. Hence, $Y=qp$-$V_{M}^{g}(Q)=qp$-$V_{M}^{g}(Q^{^{\prime }})$ due
to the maximality of $qp$-$V_{M}^{g}(Q)$. It implies that $Gr(q)=q^{^{\prime
}}$.

Conversely, let $Y\in \Phi $. Then there exists $q\in qp$-$%
V_{R}^{g}(Ann(M))$ such that $Gr(q)$ is a minimal element in $%
V_{R}^{g}(Ann(M))$ and $Y=qp$-$V_{M}^{g}(Gr(q)M)$. Since $M$ is graded
quasi-primaryful, there exists an \ element $Q\in q.Spec_{g}(M)$ such that $%
Gr((Q:_{R}M))=Gr(q)$. So, $Y=qp$-$V_{M}^{g}(Gr(q)M)=qp$-$%
V_{M}^{g}(Gr((Q:_{R}M))M)=qp$-$V_{M}^{g}(Q)$, and so $Y$ is irreducible by
Theorem \ref{T4.9}. Suppose that $Y=qp$-$V_{M}^{g}(Q)\subseteq qp$-$%
V_{M}^{g}(Q^{^{\prime }})$, where $Q^{^{\prime }}\in q.Spec_{g}(M)$. Since $%
Q\in qp$-$V_{M}^{g}(Q^{^{\prime }})$ and $Gr(q)$ is minimal, it follows that
$Gr((Q:_{R}M))=Gr((Q^{^{\prime }}:_{R}M))$. Now, by Proposition \ref{P3.2}%
(3), we have $Y=qp$-$V_{M}^{g}(Q)=qp$-$V_{M}^{g}(Gr((Q:_{R}M))M)=qp$-$%
V_{M}^{g}(Gr((Q^{^{\prime }}:_{R}M))M)=qp$-$V_{M}^{g}(Q^{^{\prime }})$.
\end{proof}

\bigskip

A topological space $(X$, $\tau )$ is said to be Noetherian if the open
subsets of $X$ satisfy the ascending chain condition. Since closed subsets
are complements of open subsets, it comes to the same thing to say that the
closed subsets of $(X$, $\tau )$ satisfy the descending chain condition.%
\newline
Let $(X$, $\tau )$ be a Noetherian topological space. Then every subspace of
$(X$, $\tau )$ is compact. In particular, $(X$, $\tau )$ is compact (see
\cite{12}).

\begin{definition}
\label{D4.12}Let $M$ be a graded $R$-module. The graded Zariski quasi
primary radical of a graded submodule $K$ of $M$, denoted by $Zqp$-$%
Gr_{M}(K) $, is the intersection of all members of $qp$-$V_{M}^{g}(K)$ for ($%
qp.Spec_{g}(M)$, $q.\tau ^{g}$), that is, $Zqp$-$Gr_{M}(K)=\bigcap%
\limits_{Q\in qp\text{-}V_{M}^{g}(K)}Q=\bigcap \{Q\in qp.Spec_{g}(M)\mid
Gr((Q:_{R}M))\supseteq Gr((K:_{R}M))\}$. We say a graded submodule $K$ is a $%
Z_{qp}$-radical submodule if $K=Zqp$-$Gr_{M}(K)$.
\end{definition}

\bigskip

We present the next important Theorem. It is one of the main results of this
paper.

\begin{theorem}
\label{T4.13}Let $M$ be a graded $R$-module. Then, ($q.Spec_{g}(M)$, $q.\tau
_{M}^{g}$) is a Noetherian topological space (and so is quasi-compact) if
and only if the $ACC$ for the graded Zariski quasi primary radical
submodules of $M$ holds.
\end{theorem}

\begin{proof}
Suppose the $ACC$ for the graded Zariski quasi primary radical submodules of
$M$. Let $qp$-$V_{M}^{g}(K_{1})\supseteq qp$-$V_{M}^{g}(K_{2})\supseteq $%
...\ be a descending chain of closed sets $qp$-$V_{M}^{g}(K_{i})$ of ($%
q.Spec_{g}(M)$, $q.\tau _{M}^{g}$), where $K_{i}$ is a graded submodule of $%
M $. Then $\Im (qp$-$V_{M}^{g}(K_{1}))=Zqp$-$Gr_{M}(K_{1})\subseteq \Im (qp$-%
$V_{M}^{g}(K_{2}))=Zqp$-$Gr_{M}(K_{2})\subseteq $...\ is an ascending chain
of graded Zariski quasi primary radical submodules of $M$. So, by
assumption, there exists $n\in
\mathbb{N}
$ such that for all $i\in
\mathbb{N}
$, $Zqp$-$Gr_{M}(K_{n})=Zqp$-$Gr_{M}(K_{n+i})$. Now, by Theorem \ref{T4.2}%
(1), $qp$-$V_{M}^{g}(K_{n})=qp$-$V_{M}^{g}(Zqp$-$Gr_{M}(K_{n}))=qp$-$%
V_{M}^{g}(K_{n+i})=qp$-$V_{M}^{g}(Zqp$-$Gr_{M}(K_{n+i}))$. Thus ($%
q.Spec_{g}(M)$, $q.\tau _{M}^{g}$) is a Noetherian topological space.
Conversely, suppose that ($q.Spec_{g}(M)$, $q.\tau _{M}^{g}$) is a
Noetherian topological space. Let $K_{1}\subseteq K_{2}\subseteq $...\ be an
ascending chain of a graded quasi primary radical submodules of $M$. Thus $%
qp $-$V_{M}^{g}(K_{1})\supseteq qp$-$V_{M}^{g}(K_{2})\supseteq $...\ be a
descending chain of closed sets $qp$-$V_{M}^{g}(K_{i})$ of ($q.Spec_{g}(M)$,
$q.\tau _{M}^{g}$). By assumption there is $n\in
\mathbb{N}
$ such that for all $i\in
\mathbb{N}
$, $qp$-$V_{M}^{g}(K_{n})=qp$-$V_{M}^{g}(K_{n+i})$. Therefore, $K_{n}=Zqp$-$%
Gr_{M}(K_{n})=\Im (qp$-$V_{M}^{g}(K_{n}))=\Im (qp$-$V_{M}^{g}(K_{n+i}))=Zqp$-%
$Gr_{M}(K_{n+i})=K_{n+i}$. Therefore the $ACC$ for the graded Zariski quasi
primary radical submodules of $M$ holds.
\end{proof}

\bigskip
Recall that every Noetherian topological space has only finitely many
irreducible components, (see \cite{**}).

As a consequence of Theorem \ref{T4.13}, we have the following Theorem.

\begin{theorem}
\label{T4.14} Let $M$ be a graded quasi-primaryful
$R$-module. If ($q.Spec_{g}(M)$, $q.\tau _{M}^{g}$) is a Noetherian space,
then we have the following:\newline
(1) If $\varphi $ is injective, then every ascending chain of
graded quasi primary submodules of $M$ is stationary.\newline
(2) $R$ has finitely many graded minimal prime ideals.
\end{theorem}

\begin{proof}
 (1) Let $K_{1}\subseteq K_{2}\subseteq $\textperiodcentered
\textperiodcentered \textperiodcentered\ be an ascending chain of graded
quasi primary submodules of $M$. Then $qp$-$V_{M}^{g}(K_{1})\supseteq qp$-$%
V_{M}^{g}(K_{2})\supseteq $\textperiodcentered \textperiodcentered
\textperiodcentered\ is a descending chain of closed subsets of ($%
q.Spec_{g}(M)$, $q.\tau _{M}^{g}$), which is stationary by assumption. There
exists an integer $n\in
\mathbb{N}
$ such that $qp$-$V_{M}^{g}(K_{n})=qp$-$V_{M}^{g}(K_{n+i})$ for each $i\in
\mathbb{N}
$. By Theorem \ref{T3.7}, we have $K_{n}=K_{n+i}$ for
each $i\in
\mathbb{N}
$. This completes the proof.\newline
(2) Since every Noetherian topological space has finitely many irreducible
components, the result follows from Corollary \ref{C4.10}.
\end{proof}

\bigskip

Spectral spaces have been characterized by Hochster \cite[p.52. Ptoposition 4%
]{8} as the topological spaces $X$ which satisfy the following conditions:%
\newline
(1) $X$ is a $T_{0}$-space.\newline
(2) $X$ is quasi-compact.\newline
(3) The quasi-compact open subsets of $X$ are closed under finite
intersection and form an open base.\newline
(4) Each irreducible closed subset of $X$ has a generic point.

\bigskip

The following theorem is one of the main result of this article. In
particular, we observe ($q.Spec_{g}(M)$, $q.\tau _{M}^{g}$) from the point
of view of spectral topological spaces.

\begin{theorem}
\label{T4.15}Let $M$ be a graded $R$-module and
the map $\psi ^{q}$ be surjective. Then the following statements are
equivalent:\newline
(1) ($q.Spec_{g}(M)$, $q.\tau _{M}^{g}$) is a spectral space.\newline
(2) ($q.Spec_{g}(M)$, $q.\tau _{M}^{g}$) is a $T_{0}$-space.\newline
(3) $\varphi $ is injective.\newline
(4) If $qp$-$V_{M}^{g}(P)=qp$-$V_{M}^{g}(Q)$, then $P=Q$, for any $P$, $Q\in
q.Spec_{g}(M)$.\newline
(5) $\mid q.Spec_{g}(M)\mid \leq 1$ for every $q\in V^{q}(Ann(M))$ with $%
Gr(q)=p$.\newline
(6) $\varphi $ is a graded $R$-homeomorphism.
\end{theorem}

\begin{proof}
 (1)$\Rightarrow $(2) Is trivial\newline
(2)$\Rightarrow $(1) Holds by combining Theorem \ref{T3.16},
Theorem \ref{T3.17} and\ Theorem \ref{T4.9}.\newline
(2)$\Leftrightarrow $(3)$\Leftrightarrow $(4)$\Leftrightarrow $(5) Follows
by Theorem \ref{T4.3}.\newline
(3)$\Leftrightarrow $(6) By Corollary \ref{T3.13}.\newline
\end{proof}


\bigskip\bigskip\bigskip\bigskip

\end{document}